\documentclass[oneside,11pt]{amsart}

\newskip\stdskip                      
\usepackage{pinlabel}  
\usepackage{pictexwd,dcpic}
\usepackage{amscd,latexsym,amssymb}
\usepackage{bm}
  \usepackage{hyperref}  
  \hypersetup{%
  bookmarksnumbered=true,%
  bookmarks=true,%
  colorlinks=true,%
  linkcolor=blue,%
  citecolor=blue,%
  filecolor=blue,%
  menucolor=blue,%
  pagecolor=blue,%
  urlcolor=blue,%
  pdfnewwindow=true,%
  pdfstartview=FitBH}
   \title{On the $KO$-groups of toric manifolds}

\author[L. Cai]{Li Cai}
\address{Department of Mathematical Sciences, Xi'an Jiaotong-Liverpool University, 111 Ren'ai Road, Dushu Lake Higher Education Town, Suzhou 215123, Jiangsu, China}
\email{Li.Cai@xjtlu.edu.cn}

\author[S. Choi]{Suyoung Choi}
\thanks{The first named author was supported by National Natural Science Foundation of China (grant no. 11801457). The second named author was supported by the National Research Foundation of Korea Grant funded by the Korean Government (NRF-2019R1A2C2010989)}
\address{Department of Mathematics, Ajou University, 206 Worldcup-ro, Suwon 16499, South Korea}
\email{schoi@ajou.ac.kr}

\author[H.Park]{Hanchul Park}
\address{Department of Mathematics Education, Jeju National University, 102 Jejudaehak-ro, Jeju-si, Jeju-do 63243, Republic of Korea}
\email{hanchulp@jejunu.ac.kr}

\newtheorem{thm}{Theorem}[section]

\newtheorem{cor}[thm]{Corollary}
\newtheorem{lem}[thm]{Lemma}
\newtheorem{prop}[thm]{Proposition}

\theoremstyle{definition}

\theoremstyle{definition}

\newtheorem{exm}[thm]{Example}

\theoremstyle{remark}

\newtheorem*{ack}{Acknowledgment}

\def\co{\colon\thinspace}

\begin{document}
\begin{abstract}    
In this paper we consider the real topological $K$-groups of a toric manifold $M$, which turns out to be closely related to the topology of the small cover $M_\mathbb{R}$, the fixed points under the canonical conjugation on $M$. Following the work of Bahri and Bendersky \cite{BB00}, we give an explicit formula for the $KO$-groups of toric manifolds, and then we characterize the two extreme classes of toric manifolds according to the two $\mathcal{A}(1)$ modules shown in \cite{BB00}. 
\end{abstract}

\maketitle

\section{Introduction}
Following Davis and Januszkiewicz \cite{DJ91}, by a toric manifold we mean a smooth manifold with locally standard torus action, such that the orbit space is homeomorphic to a simple convex polytope (see \eqref{def:M} for details). Toric manifolds include all projective non-singular toric varieties, and they generate the complex cobordism ring (see \cite{BR98}). We refer the readers to Buchstaber and Panov \cite[Chapter 5]{BP02} for more details.

Let $c\co KO^*(M)\to K^*(M)$ and 
$r\co K^*(M)\to KO^*(M)$ be the complexification and realification maps 
respectively, and let the cokernel  
\[
   W^{i}(M)=KO^{2i}(M)/r
\]
of $r$ be the $i$-th \emph{Witt group}, 
which is clearly $4$-periodic. 
Since $r\circ c=2$, Witt groups consist of only $2$-torsion elements. 

Following \cite{KOtoric} and \cite{Zib15},
consider the Bott sequence 
\[
\begin{CD}
\cdots\rightarrow 
KO^{i+1}(M)@>\cdot\eta>>KO^i(M)@>c>>K^i(M)@>r\circ t^{-1}>>
KO^{i+2}(M)\rightarrow \cdots
	\end{CD}
\]
where $t\co K^0(M)\to K^{-2}(M)$ is the isomorphism given by multiplying the Bott element $t$, and $\eta\co KO^{0}(M)\to KO^{-1}(M)$ is the map given by multiplying the generator $\eta\in KO^{-1}(pt)$ from the Hopf bundle.

Since the integral cohomology $H^*(M;\mathbb{Z})$ concentrates in even dimensions, by Atiyah-Hirzebruch spectral sequence, the complex topological $K$-group $K^0(M)$ is free and $K^1(M)=0$. The Bott sequence above splits into the form
\[
\begin{CD}
0\rightarrow KO^{2i+1}(M)\stackrel{\eta}{\rightarrow}KO^{2i}(M)
\stackrel{c}{\rightarrow} K^{2i}(M)\stackrel{r\circ t^{-1}}{\rightarrow} KO^{2i+2}(M)\stackrel{\eta}{\rightarrow} KO^{2i+1}(M)\rightarrow 0,
\end{CD}
\]
we see that $\eta$ induces an isomorphism 
\begin{equation}
	\begin{CD}
		\eta\co W^{i+1}(M)@>\cong>>KO^{2i+1}(M).
        \end{CD}
	\label{eq:odd}
\end{equation}
Since the image of $KO^{2i}(M)$ under $c$ is a subgroup of the free 
abelian group $K^{2i}(M)$, it is also free and 
the sequence
\[
  \begin{CD}
  0@>>> KO^{2i+1}(M)@>\eta>>KO^{2i}(M)@>c>>\mathrm{im}c@>>>0	\end{CD}
\]
splits. By \eqref{eq:odd} we have that 
$\eta^2\co W^{i+1}(M)\to KO^{2i}(M)$ induces an isomorphism
 onto its image and 
\begin{equation}
    KO^{2i}(M)=W^{i+1}(M)\bigoplus\text{free part}.
    \label{eq:even}
\end{equation}
Next we determine the Witt group 
\[W^*(M)=\bigoplus_{i=0}^3W^i(M).\]

Let $\mathcal{A}(1)$ be the subalgebra of the $\mathrm{mod}$ $2$ Steenrod algebra $\mathcal{A}$ generated by $Sq^1$ and $Sq^2$. Bahri and Bendersky \cite{BB00} considered the $\mathrm{mod}$ $2$ cohomology $H^*(M;\mathbb{Z}_2)$ as $\mathcal{A}(1)$ module, which turns out to be a direct sum
\begin{equation}
       H^*(M;\mathbb{Z}_2)=\underline{\mathcal{S}}\oplus\underline{\mathcal{M}}, \label{eq:dec}
\end{equation}
where 
\[\underline{S}=\oplus_i\mathbb{Z}_2\{x_i\}\] 
is a trivial $\mathcal{A}(1)$ module generated by elements $x_i$ such that $Sq^1x_i=Sq^2x_i=0$, and 
\[\underline{\mathcal{M}}=\oplus_j\mathbb{Z}_2\{y_j,z_j\}\] 
is generated by pairs of elements $y_j,z_j$ such that $Sq^1y_j=Sq^1z_j=0$ and $Sq^2y_j=z_j$. 
By considering the Adams spectral sequence converging to $ko_*(M)$, with $E_2$ term
\[\mathrm{Ext}_{\mathcal{A}(1)}^{s,t}(
H^*(M;\mathbb{Z}_2),\mathbb{Z}_2)=
\mathrm{Ext}_{\mathcal{A}(1)}^{s,t}(\underline{\mathcal{S}},\mathbb{Z}_2)
\oplus\mathrm{Ext}_{\mathcal{A}(1)}^{s,t}(\underline{\mathcal{M}},\mathbb{Z}_2),\]
they showed that the Adams spectral sequence
collapses at $E_2$-page. 
Here is a conclusion of their results, in which the free part follows easily from the Atiyah-Hirzebruch spectral sequence with rational coefficients $KO^*(pt;\mathbb{Q})$. 
\begin{thm}[\cite{BB00}]\label{thm:BB}
	All torsion elements in $KO^*(M)$ have order $2$, and they come from the module $\underline{\mathcal{S}}$. More precisely each element $x\in\underline{\mathcal{S}}$ in dimension $n_x$ contributes two $2$-torsion elements in dimensions $n_x-1$ and $n_x-2$, respectively. 
	The free part of the group $KO^{2i}(M)$ has rank $n_{2i}$, where 
	\begin{align*}
	n_0&=n_4=\mathrm{dim}_{\mathbb{Q}}
	 \bigoplus_{k\in\mathbb{Z}}H^{4k}(M;\mathbb{Q})\\
	 n_2&=n_6=\mathrm{dim}_{\mathbb{Q}}
	 \bigoplus_{k\in\mathbb{Z}}H^{4k+2}(M;\mathbb{Q}).
	 	\end{align*}
	 	\end{thm}
	
	\subsection{Main Results}In order to determine the Witt group $W^*(M)$ completely, it remains to determine the submodule $\underline{\mathcal{S}}$. We will see that it is closely related to the cohomology of the small cover $M_{\mathbb{R}}$, which is a submanifold of $M$ being the fixed points of a special involution, the ``conjugation" on $M$ (see \eqref{def:MR} for details).
	 
Here is our trivial observation, which is a special case from the conjugation spaces in the sense of Hausmann, Holm and Puppe \cite{HHP05}:
\begin{lem}\label{lem:key}
	There is a canonical isomorphism 
	\[\begin{CD}
	\phi\co H^*(M;\mathbb{Z}_2)@>\cong>>H^*(M_{\mathbb{R}};\mathbb{Z}_2),
	\end{CD}\]
	of rings such that $\phi\circ Sq^2=Sq^1\circ \phi$, where the image of an element in dimension $2i$ is in dimension $i$, for $i=0,\ldots, n$.
\end{lem}	 
Therefore we only need to consider those $Sq^1$-cocycles in $H^*(M_{\mathbb{R}};\mathbb{Z}_2)$. In this direction we have a complete answer with the help of the previous works \cite{Tre12}, \cite{ST12}, \cite{CP13}, \cite{CC16} on the cohomology groups of $M_{\mathbb{R}}$. Let $\lambda_{2}$ be the $\mathrm{mod}$ $2$ reduction of the characteristic matrix $\lambda$ of $M$ (see \eqref{eq:lambda}), and let $\mathrm{Row}(\lambda_{2})$ be the subspace of $\mathbb{Z}_2^m$ spanned by the row vectors of $\lambda_{2}$. Let $K$ be the simplicial complex
dual to the boundary of $P$, and let $K_{\omega}$ be the full subcomplex with vertex set $\{i\mid w_i=1\}$ for $\omega=(w_i)_{i=1}^m\in \mathrm{Row}(\lambda_2)$ (see \eqref{def:KP}, \eqref{def:KS}
for details).
\begin{thm}\label{thm:main}
Let $M$ be the toric manifold over a simple convex polytope $P$ with characteristic matrix $\lambda$. Let $K$ be the simplicial complex dual to the boundary of $P$. Then for each mod $4$ integers $i\in\mathbb{Z}_4$, we have additive isomorphisms
\begin{equation}
    W^{i+1}(M)= \bigoplus_{k\in\mathbb{Z}}
\bigoplus_{\omega\in\mathrm{Row}(\lambda_2)}
    \widetilde{H}^{i+4k}(K_{\omega};\mathbb{Z}_2),\label{eq:Witt}
	\end{equation}
where $\widetilde{H}^{i}(K_{\emptyset};\mathbb{Z}_2)$ is non-trivial only when $i=-1$, which is isomorphic to $\mathbb{Z}_2$. 
\end{thm}
Combining the results together, for $i\in\mathbb{Z}_4$, 
we have the following additive isomorphisms
\begin{align}
KO^{2i}(M)&=
\bigoplus_{k\in\mathbb{Z}}\left[
H^{4k+1-(-1)^i}(M;\mathbb{Z})\bigoplus
 \bigoplus_{\omega\in\mathrm{Row}(\lambda_2)}
\widetilde{H}^{i+4k}(K_{\omega};\mathbb{Z}_2)\right],\\
KO^{2i+1}(M)&= \bigoplus_{k\in\mathbb{Z}}\bigoplus_{\omega\in\mathrm{Row}(\lambda_2)}
\widetilde{H}^{i+4k}(K_{\omega};\mathbb{Z}_2). 
	\end{align}
A geometric construction of bundles realizing part of these elements is given in \cite{CR05}.
	
Next we consider two extreme cases in \eqref{eq:dec}. We say that a toric manifold $M$ is of $\mathcal{S}$-type if $H^*(M;\mathbb{Z}_2)=\underline{\mathcal{S}}$ and it is of $\mathcal{M}$-type if the reduced cohomology $\widetilde{H}^*(M;\mathbb{Z}_2)=\underline{\mathcal{M}}$. For the first case we have the following result, in which 
(S4) $\Rightarrow$ (S3) is proved in \cite{DJ91}.
\begin{thm}\label{thm:S}
Let $M$ be a toric manifold with characteristic matrix $\lambda$. 
The following are equivalent:
\begin{enumerate}
\item[(S1)] $M$ is of $\mathcal{S}$-type;
\item[(S2)] the corresponding small cover $M_{\mathbb{R}}$ 
      has no $2$-torsion elements in its integral cohomology 
      $H^*(M_{\mathbb{R}};\mathbb{Z})$; 
\item[(S3)] $H^*(M_{\mathbb{R}};\mathbb{Z})$ is torsion free;
\item[(S4)] (up to a basis change) the mod $2$ reduction of $\lambda$ has only one non-zero entry in each column (i.e., $M_{\mathbb{R}}$ is a pull-back of the linear model in the sense of \cite{DJ91}).
 \end{enumerate}
\end{thm}
There are many interesting examples in this case, 
including the isospectral manifolds of tridiagonal hermitian matrices (see \cite{BFR89}, \cite{Dav87}); the class of torsion free small covers includes Tomei manifolds \cite{Tom84}, the isospectral manifolds of tridiagonal real symmetric matrices.

For the second case, we make a conclusion here, where the 
equivalence of (M1) and (M2) is proved in \cite{BB00}.  
\begin{thm}\label{thm:M}
The following are equivalent:
\begin{enumerate}
\item[(M1)] $M$ is of $\mathcal{M}$-type;
\item[(M2)] the reduced group $\widetilde{KO}^*(M)$ is torsion free;
\item[(M3)] $\widetilde{H}^*(M_{\mathbb{R}};\mathbb{Z})$ has only 
$2$-torsion elements.
\item[(M4)] $\widetilde{H}^*(K_{\omega};\mathbb{Z}_2)=0$, for all non-zero
      $\omega\in \mathrm{Row}(\lambda_{2})$.
	\end{enumerate}
\end{thm}
For a toric manifold $M$ of $\mathcal{M}$-type, the ring structure $KO^*(M)$ is given explicitly in \cite[Theorem 6.5]{KOtoric}. It is observed in \cite[Proposition 5.1]{Par17} that if the property (M3) holds, then the dimension of $M_{\mathbb{R}}$ must be even. 

Finally, we consider the problem that how to construct toric manifolds of $\mathcal{M}$-type. Nishimura \cite{Nis08} gives some interesting examples, and a general idea to construct such kind of manifolds is given in \cite{KOtoric}, using the wedge construction given in \cite{BBCGwedge}. Following their idea, we have the result below (see Section \ref{sec:M} for details).
\begin{thm}\label{thm:J}
Let $M$ be a toric manifold over a simple convex 
polytope $P$ of dimension $n$, and let $K$ be the 
simplicial complex dual to the 
boundary of $P$. If vertices $v_1,\ldots,v_n$ of 
$K$ span a (maximal) simplex in $K$, then
the toric manifold $M(v_1,\ldots,v_n)$ obtained 
from the simplicial wedge construction is of 
$\mathcal{M}$-type.
	\end{thm}
The paper is organized as follows. The calculation of the Witt group $W^*(M)$ is done in Section \ref{sec:KO}, where Lemma \ref{lem:key} and Theorem \ref{thm:main} are proved. The characterizations of $\mathcal{S}$-type and $\mathcal{M}$-type toric manifolds (i.e., Theorem \ref{thm:S} and Theorem \ref{thm:M}) are proved in Section \ref{sec:S} and Section \ref{sec:M} respectively. Theorem \ref{thm:J} is also proved in Section \ref{sec:M}.

\section{The cohomology of small covers}\label{sec:KO}
Let $P$ be a simple convex polytope of dimension $n$ and let $\mathcal{F}=\{F_1,\ldots, F_m\}$ be the set of facets (i.e., codimension one faces) of $P$. Since $P$ is simple, each face $F\subset P$ of codimension $l$ is an intersection of exactly $l$ facets. Suppose 
      \[\lambda\co \mathcal{F}\to \mathbb{Z}^n\] 
is a \emph{characteristic function} in the sense that whenever $F_{i_1},\ldots, F_{i_n}$ intersects at a vertex, their images $\lambda(F_{i_1}),\ldots, \lambda(F_{i_n})$ span the integer lattice $\mathbb{Z}^n$ of rank $n$. Occasionally we treat $\lambda$ as an integral matrix of size $n\times m$ in the form 
      \begin{equation}
      \lambda=\begin{pmatrix}
      a_{11} &a_{12} & \cdots & a_{1m}\\
      a_{21} &a_{21} & \cdots  &a_{2m}\\
      \cdots &\cdots & \cdots & \cdots\\
      a_{n1} &a_{n2} & \cdots & a_{nm} 
      \end{pmatrix},\label{eq:lambda}
\end{equation}
and call it a \emph{characteristic matrix}.

Let $T^n=S^1\times\ldots\times S^1$ be the $n$-dimensional compact torus whose Lie algebra is generated by $e_1,\ldots, e_n$, an additive basis of the lattice $\mathbb{Z}^n$. For a subset $I\subset [m]=\{1,2,\ldots,m\}$, let $T_I\subset T^n$ be the sub-torus with its Lie algebra generated by $\{\lambda({F}_i)\}_{i\in I}$. Then the corresponding toric manifold $M$ is given by 
      \begin{equation}
            M=T^n\times P/\sim \label{def:M}
      \end{equation}
where $(t,p)\sim (t',p')$ if and only if $p=p'$ and $t't^{-1}\in T_I$ where $I=\{i\in [m]\mid p\in F_i\}$ (if $p$ is in the interior of $P$, then $I=\emptyset$ with $T_{\emptyset}$ the identity of $T^n$). 
      
Following \cite{DJ91}, let $\tau\co M\to M$ be the involution given by $\tau([t,p])=[t^{-1},p]$. The fixed points of $\tau$ is a submanifold of dimension $n$, called the small cover associated to $M$ and is denoted by $M_{\mathbb{R}}$ (if $M$ is a projective non-singular toric variety, then $\tau$ can be realized as the complex conjugation).
      
Recall that the simplicial complex dual to the boundary of $P$ is given by       
      \begin{equation}
      K_P=\{\{i_1,i_2,\ldots,i_l\}\subset [m]
      \mid
      F_{i_1}\cap F_{i_2}\cap\cdots \cap F_{i_l}
      \not=\emptyset \}. \label{def:KP}
     \end{equation}
   \begin{thm}[\cite{DJ91}]
Let $M$ be a toric manifold and $M_{\mathbb{R}}$ be the corresponding small cover. Then we have the isomorphism 
\[H^*(M;\mathbb{Z})=
  \mathbb{Z}[v_1,\ldots,v_m]/\left(I+J\right)\]
of graded rings, where the degree of each generator $v_k$ is $2$, $I$ is the Stanley-Reisner ideal generated by square-free monomials $v_{i_1}v_{i_2}\cdots v_{i_s}$ such that 
$i_1,\ldots,i_s$ does not span a simplex in $K_P$, 
and $J$ is generated by $n$ linear elements 
      	\[ \sum_{k=1}^m a_{ik}v_{k},\]
where $(a_{ik})_{k=1}^m$ is the $i$-th row of $\lambda$, $i=1,\ldots, n$.
      	       	
Likewise, we have the isomorphism
\[
  H^*(M_{\mathbb{R}};\mathbb{Z}_2)=
   \mathbb{Z}_2[x_1,\ldots,x_m]/
  \left(I_{\mathbb{R}}+J_{\mathbb{R}}\right)
 \]
 of graded rings, where the degree of each generator $x_k$ is $1$, $I_{\mathbb{R}}$ and $J_{\mathbb{R}}$ are defined as above by replacing $v_i$ with $x_i$, respectively.  
\end{thm}
\begin{proof}[Proof of Lemma \ref{lem:key}]
Let $\phi\co H^*(M;\mathbb{Z}_2)\to 
H^*(M_{\mathbb{R}};\mathbb{Z}_2)$ be the map sending each mod $2$ reduction $[v_k]$ of $v_k$ to $x_k$, $k=1,\ldots,m$. Since 
\[
     \phi(Sq^2([v_k]))=\phi([v_k]^2)=x_k^2=Sq^1x_k,
     \] 
we have $\phi\circ Sq^2=Sq^1\circ \phi$ as desired. 
\end{proof}
First note that since $H^*(M;\mathbb{Z}_2)$ is a trivial $Sq^1$-module, a pair of generators in $\underline{\mathcal{M}}$ connected by $Sq^2$ corresponds bijectively to a pair of elements in $H^*(M_{\mathbb{R}};\mathbb{Z}_2)$ connected by $Sq^1$, namely a $\mathbb{Z}_2$ summand in $H^*(M_{\mathbb{R}};\mathbb{Z})$. 

In what follows, a non-trivial element in an abelian group is \emph{primitive} if it is not divisible by an integer greater than $1$.
 To complete the proof of Theorem \ref{thm:main}, we need the following result proved in \cite[Corollary 5.3]{CC16}.
\begin{prop}\label{prop:CC}
Let $G_1^i$ be the group 
$\bigoplus_{\omega\in\mathrm{Row}(\lambda_2)}
\widetilde{H}^{i}(K_{\omega};\mathbb{Z})$ 
and let $G_2^i$ be $H^{i+1}(M_{\mathbb{R}};\mathbb{Z})$, 
in each dimension $i\geq -1$.
We have a correspondence between them, such that
\begin{enumerate}
	\item the free part of $G_1^i$ is isomorphic to that of $G_2^i$, and 
\item primitive $2^j$-torsion elements in $G_1^i$ correspond bijectively to primitive $2^{j+1}$-torsion elements in $G_2^i$, for all $j\geq 1$.
\end{enumerate} 	
\end{prop}
\subsection{Proof of Theorem \ref{thm:main}}
By Theorem \ref{thm:BB}, the Witt group $W^*(M)$ of a toric manifold comes from the module $\underline{\mathcal{S}}$. Let
\[
\underline{\mathcal{S}}=
\bigoplus_{i=0}^{\infty}\underline{\mathcal{S}}^{2i}
\]
be graded by cohomology dimensions. We claim that for each $i\geq 0$ there is an isomorphism
\begin{equation}
    \underline{\mathcal{S}}^{2i}=
    \bigoplus_{w\in\mathrm{Row}(\lambda_2)}
    \widetilde{H}^{i-1}(K_{\omega};\mathbb{Z}_2).\label{eq:claim}
\end{equation}
Let $j\geq 1$ be an integer. By Lemma \ref{lem:key}, an element as a $Sq^2$-cocycle in 
$\underline{\mathcal{S}}^{2i}$ corresponds 
bijectively to a $Sq^1$-cocycle in
$H^i(M_{\mathbb{R}};\mathbb{Z}_2)$, and under the mod $2$ reduction, such an element comes from $H^i(M_{\mathbb{R}};\mathbb{Z})$, being either 1) a torsion-free generator,  
2) a primitive $2^{j+1}$-torsion element 
in $H^i(M_{\mathbb{R}};\mathbb{Z})$ 
(after the mod $2$ reduction, a primitive $2^{j+1}$-torsion element in dimension $i$ becomes a pair of $Sq^1$-cocycles in dimensions $i$ and $i-1$, respectively) or 3) a primitive $2^{j+1}$-torsion element in $H^{i+1}(M_{\mathbb{R}};\mathbb{Z})$.

To prove the claim \eqref{eq:claim}, let $G_1'$ be the group on the right-hand side. Notice that an element in $G_1'$ comes from either 1') a torsion-free generator in $G_1^{i-1}$, 2') a primitive $2^j$-torsion element in $G_1^{i-1}$ or 3') a primitive $2^j$-torsion element in $G_1^{i}$. Comparing the two situations above, it follows from Proposition \ref{prop:CC} that we get a bijection between $\underline{\mathcal{S}}^{2i}$ and $G_1'$.

By formulas \eqref{eq:odd}, \eqref{eq:even}, the Witt group $W^{i+1}(M)$ is isomorphic to the torsion part of $KO^{2i+1}(M)$ or $KO^{2i}(M)$, where the latter is isomorphic to $\bigoplus_{k}\underline{\mathcal{S}}^{2i+2+8k}$ by Theorem \ref{thm:BB} and the $8$-periodicity of $KO$-groups. Finally, formula \eqref{eq:Witt} follows from \eqref{eq:claim}, 
and the proof is completed.
\begin{exm}
Let $P$ be the $n$-simplex with characteristic
matrix
\[
 	\lambda=
      	\bordermatrix{%
      	&1,\ldots,n & n+1 \cr
      	& I_n & * \cr
      	}
\]
of size $n\times (n+1)$, in which $I_n$ is 
the identity submatrix of size $n\times n$ and 
all entries in the last column are $-1$.
We have $M=\mathbb{C}P^n$. As the dual of 
the boundary of $P$, $K$ is 
the boundary of an $n$-simplex, thus $K_{\omega}$ 
is contractible except when $n$ is odd and 
$\omega=\sum_{i=1}^nr_i$ with $r_i$ the 
$i$-th row of $\lambda_2$.
We see that in this case $K_{\omega}=K$, and 
\[W^{n}(\mathbb{C}P^n)=
\widetilde{H}^{n-1}(K;\mathbb{Z}_2)=\mathbb{Z}_2 
\]
giving the two $2$-torsion elements 
in $\widetilde{KO}^{2n-1}(\mathbb{C}P^n)$ and $\widetilde{KO}^{2n-2}(\mathbb{C}P^n)$ 
respectively. When $n$ is even, we have 
\[
 W^*(\mathbb{C}P^n)=W^0(\mathbb{C}P^n)=
\widetilde{H}^{-1}(K;\mathbb{Z}_2)=\mathbb{Z}_2,
\]
and $\widetilde{KO}^*(\mathbb{C}P^n)$ is torsion free (see \cite{Fuj67}).
\end{exm}

\section{Toric manifolds of $\mathcal{S}$-type}\label{sec:S}
To prove Theorem \ref{thm:S}, we need a detailed information of the cell 
structure of a small cover $M_\mathbb{R}$. Recall that 
a non-trivial element is primitive if it is not divisible by an integer greater than $1$.
\begin{lem}\label{lem:XY}
Let $X$ be a CW complex with $Y\subset X$ a subcomplex. 
Suppose that each cell in either $Y$ or $X$ is a generator in their mod $2$ 
(co)homology groups, respectively. Then every primitive $2$-torsion element 
in $H^*(Y;\mathbb{Z})$ is the image of a primitive $2$-torsion element in $H^*(X;\mathbb{Z})$ under the map $i^*\co H^*(X;\mathbb{Z})\to H^*(Y;\mathbb{Z})$ induced by the inclusion $i\co Y\to X$.
\end{lem}
\begin{proof}
Since a cell of $Y$ is also a cell of $X$, by assumption, the inclusion induces a surjection $i^*_2\co H^*(X;\mathbb{Z}_2)\to H^*(Y;\mathbb{Z}_2)$ in mod $2$ cohomology. We see that a primitive $2$-torsion element $\alpha\in H^*(Y;\mathbb{Z})$ corresponds to a pair in $H^*(Y;\mathbb{Z}_2)$ connected by the Bockstein $Sq^1$, which is functorial, thus there is a pair in $H^*(X;\mathbb{Z}_2)$ connected by $Sq^1$, corresponding to a primitive $2$-torsion element $\beta\in H^*(X;\mathbb{Z})$ such that $i^*(\beta)=\alpha$.
\end{proof}

\subsection{Cellular decompositions of small covers}
In this part we recall the necessary parts we need to prove Theorem \ref{thm:S}. All results in this part are essentially known and we refer the readers to \cite{DJ91} for more details.
  
 We consider the mod $2$ reduction of the characteristic matrix \eqref{eq:lambda} as a function $\lambda_2\co \mathcal{F}\to \mathbb{Z}_2^n$. Let $\mathbb{F}^n$ be the vector space $\mathbb{Z}_2^n$. For a subset $I\subset [m]=\{1,\ldots, m\}$, let $\mathbb{F}_I^n$ be the subspace generated by the image of 
 $\{\lambda_2(F_i)\}_{i\in I}$. Then the small cover is given by
\begin{equation}
M_{\mathbb{R}}= \mathbb{F}^n\times P/\sim \label{def:MR}
\end{equation}
where $(g,p)\sim (g',p')$ if and only if $p=p'$ and $g-g'\in \mathbb{F}^n_I$ with $I=\{i\in [m]\mid p\in F_i\}$. Clearly $M_{\mathbb{R}}$ admits a piecewise linear structure induced from $P$. As an additive group, $\mathbb{F}^n$ acts on $M_\mathbb{\mathbb{R}}$ via
\[
        \mathbb{F}^n\times M_{\mathbb{R}}\to M_{\mathbb{R}}
\]
sending $(g',[g,p])$ to $[g'+g,p]$.
     
Now we embed the simple convex $n$-polytope $P$ linearly into the Euclidean space $\mathbb{R}^n$, and let 
     \begin{equation}
f\co\mathbb{R}^n\to\mathbb{R}\label{eq:f}
     	   \end{equation}
be a linear function, being generic in the sense that $f$ is injective on the set of vertices of $P$ (for example, the inner product with a vector  which is not orthogonal to any line in $\mathbb{R}^n$ connecting two vertices of $P$). In this way we have an ordering $v_1<v_2<\ldots<v_l$ of all vertices of $P$, such that $f(v_i)<f(v_j)$ if and only if $v_i<v_j$. 
     
Since $P$ is simple, for each vertex $v_i$ above there are $n$ edges connected to it. Let $E_i$ be the set of these edges, and let $E_i^-$ be the subset of $E_i$ such that $\sup_{x\in e}f(x)\leq f(v_i)$ for all $e\in E_i^-$. If we treat an edge $e\in E_i$ as a vector $\bm{e}$ pointing away from $v_i$ and suppose $\bm{v}_i$ is the position vector of $v_i$, then
\begin{equation}
        D_i=\{\bm{v}_i+\sum_{e\in E_i}t_e\bm{e}
        \mid t_e\in [0,\frac{1}{2}]  \}\label{def:D}
     \end{equation}
gives a closed neighborhood of $v_i$ in $P$. 
     
Let \[\pi\co M_{\mathbb{R}}\to P\] be the orbit map sending $[g,p]$ to $p$. Now we have a filtration 
     \begin{equation}
P_1\subset P_2\subset \cdots \subset P_l=P     
	\label{fil:P}
     \end{equation}
of polyhedra $P_i=\bigcup_{j\leq i} D_j$, giving a filtration 
\begin{equation}
      M_1\subset M_2\subset \cdots \subset M_l=M_{\mathbb{R}}
      \label{fil:M}
\end{equation}
 such that $M_i=\pi^{-1}(P_i)$.
     
Let $I=[0,1]$ be the unit interval. Recall that for $0\leq k\leq n$, an $n$-dimensional $k$-handle ($k$ is called the \emph{index} of the handle) $W$ on a piecewise linear manifold $X$ is a copy (up to piecewise linear homeomorphism) of $I^k\times I^{n-k}$, attached to the boundary $\partial X$ by a piecewise linear embedding 
\[\varphi\co(\partial I^k)\times I^{n-k} \to \partial X. \]
It turns out that the union $X\cup_{\varphi}W$ is again a piecewise linear manifold (see \cite[Chapter 6]{RS72}).
 \begin{prop}\label{prop:handle}
The filtration	\eqref{fil:M} gives a handlebody decomposition of $M_{\mathbb{R}}$. More precisely, the handlebody 
$M_{i}=M_{i-1}\cup\pi^{-1}(D_{i})$ is a piecewise linear manifold (possibly with boundary), where the index of the handle
\[\pi^{-1}(D_{i})=\mathbb{F}^n\times D_i/\sim 
\] 
 is $\mathrm{Card}E_i^-$.
\end{prop}
\begin{proof}
We prove by induction. Consider $M_1=\pi^{-1}(D_1)$, under a change of basis, we may assume that $v_1=F_1\cap F_2\cap\cdots \cap F_n$, and $\lambda(F_i)=e_i$, the $i$-th canonical basis element of $\mathbb{F}^n$, $i=1,\ldots,n$. Using coordinates $(t_e)_{e\in E_{1}}$ in \eqref{def:D} of $D_{1}$, we see that by definition \eqref{def:MR},
 \[
 \pi^{-1}D_1=\mathbb{F}^n\times D_1/\sim
 \]
is piecewise linearly homeomorphic to the cube $(t_e')_{e\in E_{1}}$, where $t_e'\in [-\frac{1}{2},\frac{1}{2}]$. 
     
Now suppose $M_{i-1}$ is a piecewise linear manifold, $2\leq i\leq n$. Let $(t_e)_{e\in E_{i}}$ be the coordinates \eqref{def:D} of $D_{i}$. Since $D_{i}$ is connected to $P_{i-1}$ by edges from $E_{i}^-$, we see that the attaching part $A_{i}$ of $D_{i}$ along $P_{i-1}$ consists of those points with coordinates $(t_e)_{e\in E_{i}}$ such that at least one component $t_e$ is $1/2$ for all $e\in E_{i}^-$. 
     
Since in the neighborhood $D_i$ of $v_i$, the coordinates are also locally standard (the images of the $n$ facets near $v_i$ under $\lambda_2$ span $\mathbb{F}^n$), up to a change of basis we can still use the coordinates of $\pi^{-1}D_1$ above, and it can be checked directly that  
\[\pi^{-1}A_i=\mathbb{F}^n\times A_i/\sim
\] 
being the attaching part of $\pi^{-1}D_i$ along $M_{i-1}$, is in the form $\partial(I^k)\times I^{n-k}$, where $k=\mathrm{Card}E_i^-$. Therefore by induction, $M_i$ is a piecewise linear manifold and the proof is completed. 
\end{proof}
\begin{prop}\label{prop:cell}
Up to homotopy, $M_{\mathbb{R}}$ admits a CW decomposition
\[
M_1\subset M_2 \subset \cdots\subset M_l=M_{\mathbb{R}}
\]
in which $M_i$ is obtained from $M_{i-1}$ by attaching a cell $e_i$ of dimension $k_i=\mathrm{Card}E_i^-$, $i=1,\ldots,l$. If $H^*(M_s;\mathbb{Z})$ has a primitive $2$-torsion element, then so does $H^*(M_{\mathbb{R}};\mathbb{Z})$.
\end{prop}
\begin{proof}
It is well known that, up to homotopy, attaching a $k$-handle is equivalent to attaching a $k$-cell. Therefore the first statement follows from Proposition \ref{prop:handle} above. In \cite[Theorem 3.1]{DJ91} it is proved that each $e_i$ is a generator in $H^*(M_{\mathbb{R}};\mathbb{Z}_2)$, therefore the second statement follows from Lemma \ref{lem:XY}.
 \end{proof}

\begin{lem}\label{lem:F}
Let $F$ be any given facet of $P$ with $s$ vertices. Then we can choose a generic linear function $f$ (see \eqref{eq:f}) such that in the ordering $v_1<v_2<\ldots<v_l$ of vertices of $P$ induced from $f$, the first $s$ vertices $v_1,\ldots,v_s$ are that of $F$.
\end{lem}
\begin{proof}
Suppose $v_0$ is a vertex of $F$. Let $P$ be embedded linearly into $\mathbb{R}^n$ such that:
\begin{enumerate}
\item	$v_0$ coincides with the origin;
\item the $n$ edges connecting to $v_0$ coincide with the $n$ coordinate axes; 
 \item $P$ lies in the first orthant 
\[
\{(x_i)_{i=1}^n\mid x_i\geq 0, i=1,\ldots,n\}
\]  
and $F$ lies in the hyperplane $\{(x_i)_{i=1}^n\mid x_n=0\}$ such that any other vertice of $P$ has the last coordinate $x_n>0$.
\end{enumerate}
Now we choose the $n$-th basis vector $\bm{e}=(0,0,\ldots,0,1)$. The function $f\co \mathbb{R}^n\to\mathbb{R}$ given by the standard inner product with $\bm{e}$ has the property that $f(v)=0$ for all vertices $v$ of $F$, and $f(v')\geq \delta>0$ for all vertices $v'$ not in $F$. We may perturb $\bm{e}$ into the form $\bm{e}_\epsilon=(\varepsilon_1, \varepsilon_2,\ldots,\varepsilon_{n-1},1)$ so that the new function $f_{\varepsilon}$ given by inner product with $\bm{e}_\varepsilon$ is generic, and $f_{\varepsilon}(v)<\frac{\delta}{2}$ (resp. $f_{\varepsilon}(v')>\frac{\delta}{2}$) for all vertices $v$ of $F$ (resp. for all vertices $v'$ not in $F$). Clearly the function $f_{\varepsilon}$ has the desired property.  
\end{proof}

 \begin{prop}\label{prop:F}
Let $v_1<\ldots<v_l$ be an ordering of vertices of $P$ induced by a generic linear function $f$, 
in which $v_1, \ldots, v_s$ are all vertices of a facet $F$. 
Let $M_s$ be the handlebody in the filtration \eqref{fil:M}. Then we have a fiber bundle 
      	\begin{equation} \begin{CD}
 I@>>> M_s@>>>\pi^{-1}F
   	\end{CD} \label{eq:fiber}
        \end{equation}
with $I=[-1,1]$ the closed interval. In particular, $M_s$ is homotopy equivalent to $\pi^{-1}F$. Moreover, $\pi^{-1}F$ is a small cover of dimension $n-1$. 
\end{prop}
\begin{proof}
Consider the polyhedron $P_s$ in the filtration \eqref{fil:P}. By definition, $F$ is a facet of $P_s$, and $M_s=\pi^{-1}P_s$. Let $F, F_1,\ldots,F_r$ be all facets of $P$ whose intersection with $P_s$ is non-empty. It is easy to see that $f_i=F_i\cap F$, $i=1,\ldots,r$, give all facets of $F$, and $P_s$ is a tubular neighborhood of $F$ in $P$. More precisely, there is a piecewise linear projection $\rho\co P_s\to F$ such that for all $x\in F$, $\rho^{-1}(x)$ is canonically identified with the closed interval $[0,1]$, and 
\begin{equation}
 \rho^{-1}f_i=F_i, \quad i=1,\ldots,r.
 \label{eq:fF}	
 \end{equation}
  
By construction \eqref{def:MR}, we have
\[
 M_s=\pi^{-1}P_s=\mathbb{F}^n\times P_s/\sim,  
 \quad 
 \pi^{-1}F=\mathbb{F}^n\times F/\sim.
\]
Let $\widetilde{\rho}\co M_s\to \pi^{-1}F$ be the map sending $[g,p]$ to $[g,\rho(p)]$. By \eqref{eq:fF}, $\widetilde{\rho}$ is a well-defined surjection, and it can be checked that $\pi^{-1}([g,x])$ is the disjoint union of two pieces of intervals $[g,\rho^{-1}(x)]$ and $[g+\lambda_2(F),\rho^{-1}(x)]$, glued along their common boundary $[g,x]$, therefore $\pi^{-1}([g,x])$ is canonically identified with the interval $I$.  As a conclusion, we have the fiber bundle $\eqref{eq:fiber}$ as desired.
      
It remains to show that $\pi^{-1}F$ is a small cover. Let $\mathbb{F}\cong\mathbb{Z}_2$ be the subspace of $\mathbb{F}^n$ generated by the element $\lambda_2(F)$, and let $\mathbb{F}^{n-1}$ be the quotient $\mathbb{F}^n/\mathbb{F}$. Since every point $[g,p]\in \pi^{-1}F$ is identified with $[g+\lambda_2(F),p]$, we see that
\[
 \pi^{-1}F= \mathbb{F}^n\times F/\sim=
\left(\mathbb{F}\oplus\mathbb{F}^{n-1}\right)\times F/\sim=
\mathbb{F}^{n-1}\times F/\sim', 
\]
where $(g,p)\sim (g',p')$ if and only if $p=p'$ and $g-g'\in \mathbb{F}^{n-1}_I$, where $\mathbb{F}^{n-1}_I$ is generated by $\{\lambda_2(F_i)\mid p\in F_i\}$ in the quotient group $\mathbb{F}^{n-1}$.
\end{proof}
The corollary below follows directly from the proof above.
\begin{cor}\label{cor:M}
Let $F, F_1,\ldots F_r$ be facets of $P$ such that $f_i=F_i\cap F$ give all facets of $F$,$i=1,\ldots,r$. Suppose that when restricted to $\{F_i\}_{i=1}^r$, the characteristic matrix $\lambda_2$ has the following form 
\begin{equation}
\lambda_2|_{P_s}=\bordermatrix{%
 &F & F_1 & F_2 & \cdots & F_r \cr
 &1 &a_{11}& a_{12} & \cdots & a_{1r} \cr
 &0 &a_{21}& a_{22} & \cdots & a_{2r} \cr
 &\cdots &\cdots &\cdots &\cdots &\cdots \cr
 &0 &a_{n1}& a_{n2} & \cdots & a_{nr} \cr
 }\label{M:P}
\end{equation} in which $\lambda_2(F)=e_1$ is the canonical basis element. Then the small cover $\pi^{-1}F$ has the following characteristic matrix
\begin{equation}
\lambda_2|_{F} =\bordermatrix{%
 & f_1 & f_2 & \cdots & f_r \cr
 &a_{21}& a_{22} & \cdots & a_{2r} \cr
 &\cdots &\cdots &\cdots &\cdots \cr
 &a_{n1}& a_{n2} & \cdots & a_{nr} \cr
}.\label{M:f}
\end{equation}
by removing the first row and the first column of \eqref{M:P}. More generally, if $\lambda(F)$ has only one non-zero entry in the $i$-th row, then the characteristic matrix of $\pi^{-1}F$ is obtained by removing the column $\lambda(F)$ together with the $i$-th row. 
\end{cor}
\subsection{Proof of Theorem \ref{thm:S}}
By Lemma \ref{lem:key}, (S1) is equivalent to the condition that
there are no pairs in $H^*(M_{\mathbb{R}};\mathbb{Z}_2)$ connected by $Sq^1$, 
namely there are no primitive $2$-torsion elements in $H^*(M_{\mathbb{R}};\mathbb{Z})$. 
We write the latter condition by (S1'). 
Since (S4) $\Rightarrow$ (S3) is proved in \cite{DJ91} and (S3) $\Rightarrow$ 
(S2) $\Rightarrow$ (S1') is obvious, we only need to prove (S1') $\Rightarrow$ (S4). 
  
The proof is an induction on the dimension $n$. A $1$-dimensional small cover is a circle, and (S4) holds trivially. When $n=2$, (S1') implies that $M_{\mathbb{R}}$ is orientable, and (S4) follows from the classification of small covers in dimension $2$ (see \cite [p. 427, Example 1.20]{DJ91}).
  
Consider $n=3$. Let $v,v'$ be two vertices connected by an edge $e$. 
Up to a basis change of $\mathbb{F}^n$ and relabel of facets, suppose $v=F\cap F_1\cap F_2$, 
$v'=F\cap F_1\cap F_3$, $e=F\cap F_1$ and $\lambda(F),\lambda(F_1),\lambda(F_2)$ form the 
standard basis of $\mathbb{F}^3$. 
Observe that by the non-degeneracy of $\lambda_2$, near $e$ it has the following form
\begin{equation} 
\bordermatrix{%
 &F & F_1 & F_2 &  F_3 \cr
 &1 & 0 &  0 & a_{13} \cr
 &0 & 1  & 0 & a_{23} \cr
 &0 &0& 1 &  1  \cr
 }.\label{eq:standard}
\end{equation}
We will show that $a_{13}=a_{23}=0$.  
  
By Lemma \ref{lem:F}, starting from the facet $F$, we get a filtration 
\[
   M_1\subset M_2\subset \cdots \subset M_s\subset \cdots\subset 
   M_l=M_{\mathbb{R}}
  \] 
as a handle decomposition of $M_{\mathbb{R}}$, in which $M_s$ is homotopy equivalent to the $2$-dimensional small cover $\pi^{-1}F$. We see that when restricted to $\{F_i\}_{i=1}^r$, $\lambda_2|_{P_s}$ has the form \eqref{M:P} and $\pi^{-1}F$ has its characteristic matrix in the form \eqref{M:f}. Notice that 
\[
H^*(\pi^{-1}F;\mathbb{Z})= H^*(M_s;\mathbb{Z})
\] 
does not have primitive $2$-torsion elements, otherwise so would $H^*(M_{\mathbb{R}};\mathbb{Z})$ by Proposition \ref{prop:cell}, a contradiction to (S1'). As a conclusion, by induction hypothesis, we have $a_{23}=0$ and
\[
 \lambda_2|_F=
 \bordermatrix{%
 & f_1 & f_2 & f_3 &\cdots & f_{r-1}&f_r \cr
 & 1 &   0 &  0 &\cdots   &    *   &* \cr
 & 0 &   1 &  1 &\cdots &    *  &* \cr
 }
\]
where there is only one non-zero entry in each column. Therefore by Corollary \ref{cor:M},
\[
\lambda_2|_{P_s}=
\bordermatrix{%
&F & F_1 & F_2 &  F_3&\cdots     &F_r \cr
&1 & 0 &  0 & a_{13}&\cdots & a_{1r} \cr
&0 & 1  & 0 & 0&\cdots &   * \cr
&0 &0& 1 &  1&\cdots & * \cr
}.
\]
It remains to show that $a_{13}=0$. Now replacing $F$ by $F_1$, we use Lemma \ref{lem:F} again to get a filtration (i.e., a different cell decomposition of $M_{\mathbb{R}}$)
\[
M_{1}'\subset M_2'\subset\cdots M_{s_1}'\subset
\cdots\subset M_{l}'=M_{\mathbb{R}}
\]
with $M_{s_1}'=\pi^{-1}P_{s_1}'$ homotopy equivalent to the $2$-dimensional small cover $\pi^{-1}F_1$, which has no primitive $2$-torsion elements by Proposition \ref{prop:cell}.  By Corollary \ref{cor:M}, the characteristic matrix of $\pi^{-1}F_1$ is in the form
\[
\lambda_{F_1}=\bordermatrix{%
& f' & f_2' & f_3' &\cdots &  \cr
& 1 &   0 &  a_{13} &\cdots   &  \cr
& 0 &   1 &  1 &\cdots &    \cr
}
\]  
where $f'=F\cap F_1$, $f_2'=F_2\cap F_1$ and $f_3'=F_3\cap F_1$, thus we get $a_{13}=0$ by induction hypothesis.

As a conclusion, for any pair $v,v'$ of vertices connected by an edge $e$, in the $4$ columns of $\lambda_2$ near $e$, if the three of them corresponding to $v$ have only one non-zero entry, so does the $4$-th one. In this way we start from a chosen vertex $v$ (where $\lambda_2$ is standard) and run through all other vertices connected to $v$ by edges, we see that (S4) holds when $n=3$.

In a higher dimension $n$ the proof is similar. Given any two vertices $v,v'$ of $P$ connected by an edge $e=F\cap F_1\cap\ldots \cap F_{n-2}$, where $v=F\cap F_1\cap\ldots\cap F_{n-1}$ and  $v'=F\cap F_1\cap\ldots \cap F_{n-2}\cap F_n$, we assume that near $e$, $\lambda_2$ is in the form
\begin{equation} 
\bordermatrix{%
&F & F_1  &\cdots &F_{n-2} & F_{n-1} & F_n \cr
&1 & 0  & \cdots &  0      &0        & a_{n1}\cr
&0 & 1  &  \cdots & 0      &0        &  *\cr
&0 & 0  & \cdots & 0 & 0 &  * \cr
&\cdots  &\cdots &\cdots &\cdots &\cdots &\cdots\cr
&0 &0&   \cdots & 1 & 0 & *\cr
&0 &0 &  \cdots & 0 & 1 & 1\cr
}.\label{eq:higher}
\end{equation} 
By induction hypothesis and repeating the argument above, we can show that the last column has only one non-zero entry. Therefore in the $n+1$ columns of $\lambda_2$ associated to any pair of vertices connected by an edge, if $n$ of them corresponding to one of the vertices have only one non-zero entry, so does the last column. Since all vertices are connected by edges, it follows that near all edges $\lambda_2$ has only one non-zero entry in each column, namely (S4) holds. The proof is completed.

\section{Toric manifolds of $\mathcal{M}$-type}\label{sec:M}
\begin{proof}[Proof of Theorem \ref{thm:M}]
The equivalence of (M1) and (M2) is given in 
Theorem \ref{thm:BB}. By Lemma \ref{lem:key},  
(M1) is equivalent to the condition that, the $Sq^1$-cohomology of $\widetilde{H}^*(M_{\mathbb{R}};\mathbb{Z}_2)$ is trivial, which is (M3). The equivalence of (M2) and (M4) is given by Theorem \ref{thm:main}.
\end{proof}

The remaining part of this section is devoted to a proof of Theorem \ref{thm:J}.  
\subsection{Toric manifolds from simplicial 
wedge constructions \cite{BBCGwedge}}
Let $K\subset 2^V$ be an abstract simplicial complex with vertex set $V$. For a simplex $\sigma\in K$, the \emph{link} of $\sigma$ is the subcomplex 
\[
\mathrm{Link}_K\sigma=\{\tau\in K\mid 
\sigma\cup\tau\in K, \ 
\sigma\cap\tau=\emptyset\}
\]
and the \emph{join} of two simplicial complexes $K_1$, $K_2$ with disjoint vertex sets $V_1$, $V_2$ 
respectively, is given by
\[
K_1*K_1=\{\sigma\subset V_1\cup V_2\mid 
\sigma=\sigma_1\cup\sigma_2, \ \sigma_1\in K_1, 
\ \sigma_2\in K_2
\}.
\] 
For a subset $S\subset V$, let $K_S$ be the \emph{full subcomplex} 
\begin{equation}
K_S=\{\sigma\subset S\mid \sigma \in K\}.
\label{def:KS}
\end{equation}
Suppose the vertex set of $K$ is $V=\{v_i\}_{i=1}^m$. The \emph{simplicial wedge construction} on the $j$-th vertex $v_j$ is a a new simplicial complex $K(v_j)$ with vertex set $\{v_i\}_{i=1}^m\cup\{v_j'\}$, which is given by 
\begin{equation}
	 K(v_j)=\{v_j,v_j'\}*\mathrm{Link}_K(v_j)\cup
	 \{v_j\}*K_{V\setminus \{v_j\}}\cup
	 \{v_j'\}*K_{V\setminus \{v_j\}.}
	 \label{def:Kv}
\end{equation}
For a simple convex polytope $P$ with facet set $\{F_i\}_{i=1}^m$, let $K_P\subset 2^{V}$ be the dual to the boundary of $P$, namely $V=\{v_i\}_{i=1}^m$ and $\{v_{i_t}\}_t\in K$ if and only if $\cap_t F_{i_t}$ is non-empty. 
	
Let $M$ be a toric $2n$-manifold over $P$ with characteristic matrix \eqref{eq:lambda}. It is proved in \cite{BBCGwedge} that there is a new toric manifold $M(v_j)$ of dimension $2n+2$ over a simple convex polytope $P(v_j)$ with its boundary dual to $K_P(v_j)$ (see \cite{GLdM13} for an explicit construction from $P$ to $P(v_j)$), and the characteristic matrix of $M(v_j)$ is
\begin{equation}
 \lambda(v_j)=\bordermatrix{%
 &v_1 & v_2  & \cdots & v_{j-1} & v_j & 
 v_{j+1} & \cdots & v_m & v_j' \cr
  &a_{11}& a_{12} & \cdots & 
 a_{1,j-1} & a_{1j} & a_{1,j+1} 
  &\cdots & a_{1m} & 0 \cr
  &a_{21}& a_{22} & \cdots & 
   a_{2,j-1} & a_{2j} & a_{2,j+1} 
   &\cdots & a_{2m} & 0 \cr
   &\cdots & \cdots & \cdots & 
   \cdots & \cdots & \cdots 
   &\cdots & \cdots & \cdots \cr 
   &a_{n1}& a_{n2} & \cdots & 
   a_{n,j-1} & a_{nj} & a_{n,j+1} 
   &\cdots & a_{nm} & 0 \cr
   & 0& 0 & \cdots & 
   0 & -1 & 0 
   &\cdots & 0 & 1 \cr
  },\label{M:P+}
\end{equation}
which is obtained from $\lambda$ by adding the last row and the last column with non-zero entries $\pm 1$.  
\subsection{Proof of Theorem \ref{thm:J}}
Suppose $v_1,\ldots,v_n$ span a simplex in $K_P$. Let $M(v_1,\ldots,v_n)$ (resp. $M_\mathbb{R}(v_1,\ldots,v_n)$) be the toric manifold (resp. small cover) obtained by consecutive simplicial wedge constructions from $M$ (resp. $M_{\mathbb{R}}$). After a basis change and the mod $2$ reduction, 
the characteristic matrix $\lambda_2(v_1,\ldots,v_n)$ of $M_{\mathbb{R}}(v_1,\ldots,v_n)$ 
is in the form
\begin{equation}
 \lambda_2(v_1,\ldots,v_n)=
 \bordermatrix{%
  &v_1,\ldots,v_n & v_{n+1},\ldots,
  v_m & v_1'\ldots,v_n' \cr
  & I_n & * & \bm{0}\cr
  & I_n & \bm{0} & I_n \cr
 }\label{M:P++}
\end{equation}
with size $2n\times (m+n)$, 
in which $I_n$ the identity matrix of size $n\times n$.  
      	
Let $\omega'=(w'_i)_{i=1}^{m+n}$ be a non-zero vector from the row space $\mathrm{Row}(\lambda_2(v_1,\ldots,v_n))$ spanned by the rows $r_1,\ldots, r_{2n}$ of $\lambda_2(v_1,\ldots,v_n)$, and let $K_{\omega'}'$ be the full subcomplex of $K'=K(v_1,\ldots,v_n)$. By definition, the vertex set of $K'_{\omega'}$ is 
\[
\{v_i\mid w_i'=1,\ 1\leq i\leq m\}\cup\{v_j'\mid w_j'=1,\ m+1\leq j\leq m+n \}.
\]
By Theorem \ref{thm:M}, we only need to show that $K'_{\omega'}$ is contractible.

By construction \eqref{def:Kv}, for the $n$ pairs of vertices $\{v_i,v_i'\}_{i=1}^n$, if the full subcomplex $K_{\omega'}'$ contains only one vertex $v\in \{v_i,v_i'\}$ but does not contain the whole edge $\{v_i,v_i'\}$, then it is a cone with apex $v$, which is contractible. If we write $\omega'=\omega_{\alpha}'+\omega_{\beta}'$, where $\omega_{\alpha}'$ (resp. $\omega_{\beta}'$) is a linear sum of $r_1,\ldots,r_n$ (resp. $r_{n+1}, \ldots, r_{2n}$), we see that this is the case, except when $\omega_{\alpha}'$ is zero. However, in the last case, $K_{\omega'}'$ is the simplex spanned by $\{v_i,v_i'\mid w_i'=1\}$, which is clearly contractible.
\begin{ack}
The authors would like to thank Tony Bahri for the introduction and discussions on the $KO$-theory of toric manifolds, and thank Anton Ayzenberg for the discussions on torsion-free small covers.
L. Cai was supported by National Natural Science Foundation of China 
(grant no. 11801457). S. Choi was supported by the National Research Foundation 
of Korea Grant funded by the Korean Government (NRF-2019R1A2C2010989). H. Park was 
supported by the National Research Foundation of Korea Grant funded by the Korean 
Government (NRF-2019R1G1A1007862).
\end{ack}


\begin{thebibliography}{99}
\bibitem{KOtoric}
L.~Astey, A.~Bahri, M.~Bendersky, F.~Cohen, 
D.~Davis, S.~Gitler, M.~Mahowald, N.~Ray 
and R.~Wood,
\emph{The $KO^*$-rings of $BT^m$, the Davis-Januszkiewicz spaces and certain toric manifolds,} Journal of Pure and Applied Algebra, \textbf{218} (2014), 303--320.

\bibitem{BB00}
A.~Bahri and M.~Bendersky, \emph{The $KO$-theory of toric manifolds,} Trans. Amer. Math. Soc.  \textbf{352} (2000), 1101--1202.

\bibitem{BBCGwedge}
A.~Bahri, M.~Bendersky, F.~R.~Cohen and S.~Gitler, \emph{Operations
on polyhedral products and a new topological construction of infinite families
of toric manifolds}, Homology, Homotopy and Applications, \textbf{17} (2015), 137--160.

\bibitem{BFR89} 
A.~M.~Bloch, H.~Flaschka and T.~Ratiu, \emph{A convexity theorem for isospectral manifolds of Jacobi matrices in a compact Lie algebra}, Duke Math. J. \textbf{61} (1990), 41--65.

\bibitem{BP02} 
V.~Buchstaber and T.~Panov, {\em  Torus Actions and Their Applications in Topology and
Combinatorics}, AMS University Lecture Series, vol.~\textbf{24} (2002).


\bibitem{BR98}
V.~M.~Buchstaber and N.~Ray, \emph{Toric manifolds and complex cobordism}, Uspekhi Mat. Nauk \textbf{53}, no. 2 (1998), 139--140; English translation: Russ. Math. Surv. \textbf{53}, no. 2 (1998), 371--373.
 
\bibitem{CC16}
L.~Cai and S.~Choi, \emph{Integral cohomology groups of real toric manifolds and small covers}, preprint (2016). Online at https://arxiv.org/pdf/1604.06988.pdf
 
\bibitem{CP13}
S.~Choi and H.~Park, 
\emph{On the cohomology and their torsion of real toric objects}, Forum Math. \textbf{29} (2017), 
no. 3, 543--553.

\bibitem{CR05}
Y.~Civan and N.~Ray, \emph{Homotopy decompositions
and real $K$-theory of Bott towers}, 
$K$-Theory \textbf{34} (2005), 1--33.

\bibitem{Dav87}
M.~W.~Davis, \emph{Some aspherical manifolds}, Duke Math. J. \textbf{55} (1987), 293--324.


\bibitem{DJ91}
M.~Davis and T.~Januszkiewicz, \emph{Convex polytopes, Coxeter orbifolds and torus actions}, Duke Math. J., \textbf{62} (1991), 417--451. 

\bibitem{Fuj67}
M.~Fujii, \emph{$K_O$-groups of projective spaces}, Osaka J. Math. \textbf{4} (1967), 141--149.	

\bibitem{GLdM13}
S.~Gitler and S.~L\'{o}pez de Medrano, \emph{Intersections of quadrics, moment-angle manifolds and connected sums,} Geom. Topol. \textbf{17} (2013), 1497--1534.

\bibitem{HHP05}
J.~C.~Hausmann, T.~S.~Holm and V.~Puppe, \emph{Conjugation spaces}, Algebraic \& Geometric Topology \textbf{5}(3) (2005), 923--964.

\bibitem{Nis08}
Y.~Nishimura, \emph{The quasi $KO$-types of toric manifolds}, In Proceedings of the International Conference in Toric Topology; Osaka City University 2006, 287--292. Edited by Megumi Harada, Yael Karshon, Mikiya Masuda and Taras Panov. Contemporary Mathematics \textbf{460} AMS, Providence RI, 2008.

\bibitem{Par17}
H.~Park, \emph{Wedge operations and doubling operations of real toric manifolds}, Chin. Ann. Math. Ser. B \textbf{38}(6) (2017), 1321--1334.

\bibitem{RS72} C.~P.~Rourke and B.~J.~Sanderson, 
  {\em Introduction to Piecewise-Linear Topology}, 
  Springer-Verlag, Berlin, 1972.

\bibitem{ST12}
A.~I.~Suciu and A.~Trevisan, \emph{Real toric varieties and abelian covers of generalized Davis-Januszkiewicz spaces}, preprint 2012.

\bibitem{Tre12}
A.~Trevisan, \emph{Generalized Davis-Januszkiewicz spaces and their applications in algebra and topology}, Ph.D. thesis, Vrije University Amsterdam (2012). Online at http://dspace.ubvu.vu.nl/handle/1871/328235

\bibitem{Tom84}
C.~Tomei, \emph{The topology of isospectral manifolds of tridiagonal matrices}, Duke Math. J. \textbf{51} (1984), 981--996.

\bibitem{Zib15}
M.~Zibrowius, \emph{$KO$-rings of full flag varieties}, Trans. Amer. Math. Soc. \textbf{367} (2015), 2997--3016.


\end{thebibliography}
\end{document}